\documentclass{amsart}
\usepackage[english]{babel}
\usepackage[latin1]{inputenc}
\usepackage[dvips,final]{graphics}
\usepackage{amsmath,amsfonts,amssymb,amsthm,amscd,array,stmaryrd,mathrsfs}
\usepackage{pstricks}
 \usepackage[all]{xy}
 \usepackage{url}
\usepackage{textcomp}
 \usepackage[final]{epsfig}
\theoremstyle{plain}
\newtheorem{thm}{Theorem}
\newtheorem{conj}{Conjecture}
\newtheorem{lem}{Lemma}[section]

\newtheorem{prop}[lem]{Proposition}
\newtheorem{theo}{Theorem}

\theoremstyle{definition}
\newtheorem{defn}[lem]{Definition}
\newtheorem{rem}[lem]{Remark}
\newtheorem{ex}[lem]{Example}
\newtheorem{ques}{Question}

\let\ssection=\section
\renewcommand{\section}{\setcounter{equation}{0}\ssection}

\newcommand{\Ab}{\mathbb{A}}
\newcommand{\R}{\mathbb{R}}
\newcommand{\Z}{\mathbb{Z}}
\newcommand{\C}{\mathbb{C}}

\newcommand{\bc}{\mathbf{c}}
\newcommand{\bd}{\mathbf{d}}
\newcommand{\be}{\mathbf{e}}
\newcommand{\bx}{\mathbf{x}}
\newcommand{\ev}{\mathbf{ev}}
\newcommand{\bev}{\overline{\mathbf{ev}}}

\newcommand{\A}{\mathcal{A}}

\newcommand{\cC}{\mathcal{C}}

\newcommand{\F}{\mathcal{F}}

\newcommand{\cM}{\mathcal{M}}

\newcommand{\Qc}{\mathcal{Q}} 
\newcommand{\Rc}{\mathcal{R}}

\newcommand{\id}{\textup{Id}}

\newcommand{\SL}{\mathrm{SL}}

\newcommand{\half}{\frac{1}{2}}
\newcommand{\thalf}{\frac{3}{2}}

\def\a{\alpha}

\def\e{\varepsilon}

\def\s{\sigma}
\def\t{\tau}

\begin{document}

\title[]{Arithmetics of 2-friezes}

\author{Sophie Morier-Genoud}

\address{Sophie Morier-Genoud,
Institut de Math\'ematiques de Jussieu,
UMR 7586,
Universit\'e Pierre et Marie Curie,
4 place Jussieu, case 247,
75252 Paris Cedex 05
}

\email{sophiemg@math.jussieu.fr
}

\date{}



\begin{abstract}
We consider the variant of Coxeter-Conway frieze patterns called 2-frieze.
We prove that there exist infinitely many closed integral 2-friezes (i.e. containing only positive integers) provided the width of the array is bigger than 4. We introduce operations on the integral 2-friezes generating bigger or smaller closed integral 2-friezes.
\end{abstract}

\maketitle



\section{Introduction}
A frieze is a finite or infinite array whose entries (that can be integers, real numbers or more generally elements in a ring)
satisfying
a local rule.
The most classical friezes are the ones introduced by Coxeter \cite{Cox}, 
and later studied by Conway and Coxeter \cite{CoCo},
for which the rule is the following:
\textit{every four neighboring entries form a matrix of determinant 1}.

\begin{figure}[hbtp]
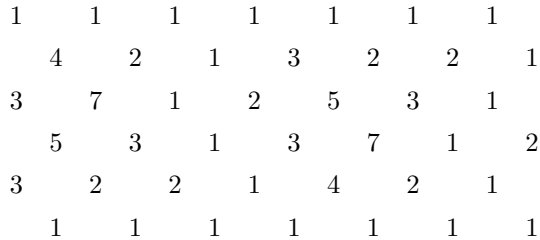

$$
 \begin{array}{ccccccccccccccc}
1&&1&& 1&&1&&1&&1&&1&& \\[4pt]
&4&&2&&1&&3&&2&&2&&1&
 \\[4pt]
3&&7&&1&&2&&5&&3&&1&\
 \\[4pt]
 &5&&3&&1&&3&&7&&1&&2\\[4pt]
 3&&2&&2&&1&&4&&2&&1\\[4pt]
&1&&1&&1&&1&&1&&1&&1
\end{array}
$$
\caption{Fragment of integral frieze of Coxeter-Conway of width 4}
\label{exCoCox}
\end{figure}

Figure \ref{exCoCox} gives an example of a Coxeter-Conway frieze  filled in with positive integers. 
In this example one can easily check that the local rule is satisfied:
$$
 \begin{array}{ccccccc}
 &B&\\
 A&&D\\
 &C&
\end{array}\quad
\Longrightarrow
\quad
AD-BC=1.
$$

A particularly interesting class of friezes is the class of \textit{integral closed friezes}.
\textit{Closed} means that the array is bounded above and below by rows of 1s, in this case we call
\textit{width} of the frieze the number of rows strictly between the top and bottom rows of 1s.
\textit{Integral} means that the frieze is filled in with positive integers.

Let us mention the following remarkable properties for Coxeter-Conway closed friezes,
 \cite{CoCo}.
 {\it
\begin{enumerate}
\item[(CC1)] Every row in a closed frieze of width $n-3$ is $n$-periodic,
\item[(CC2)] Integral closed friezes of width $n-3$ are in one-to-one correspondence with the triangulations of an $n$-gon.
The first non-trivial row
 in the frieze gives the number of incident triangles at each vertex (enumerating in a cyclic order), see Figure \ref{triang}.
\end{enumerate}}

\begin{figure}[hbtp]
\includegraphics[width=4cm]{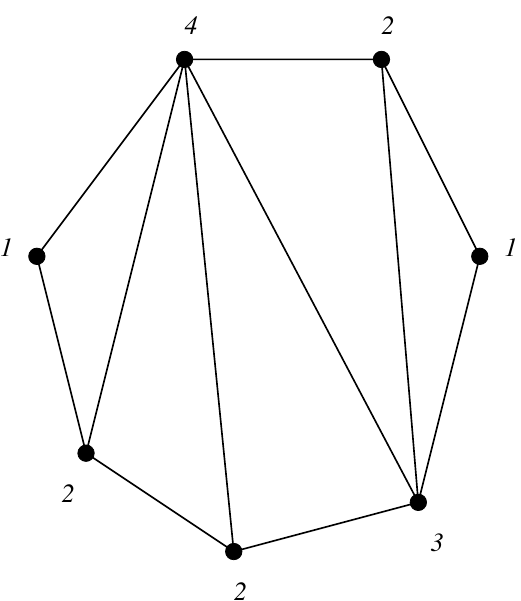}
\caption{Triangulation associated to the frieze of Figure \ref{exCoCox}}
The label attached at each vertex is the number of incident triangles.
\label{triang}
\end{figure}

Coxeter and Conway established many other surprising connections between frieze patterns and classical objects in mathematics, 
like Gauss {\it pentagramma mirificum}, Fibonacci numbers, Farey sequences\footnote{See also Richard Schwartz' applet at
http://www.math.brown.edu/$\sim$res/Java/Frieze/Main.html}, continued fractions,.... 

The study of frieze patterns is currently reviving due to connections with Fomin-Zelevinsky's cluster algebras.
This new strong interest started in 2005 with the work of Caldero and Chapoton~\cite{CaCh} 
where they connected Coxeter-Conway frieze patterns
to cluster algebras of type A. 
New versions of frieze patterns have been introduced to extend this connection to some other types \cite{BaMa}, \cite{ARS}, 
and provide new information on cluster variables, see also 
\cite{AD}, \cite{KeSc},~\cite{ADSS}.

In 2005, J. Propp suggested a variant of frieze \cite{Pro}. This variant is called \textit{$2$-frieze} in \cite{MOT}. 
The defining local rule for the variant of 2-frieze is the following:
\textit{each entry in the frieze is equal to the determinant of the matrix formed by its four neighbors}.

\begin{figure}[hbtp]
$$
\begin{array}{cccccccccccccccccc}
1&1&1&1&1&1&1&1&1&1&1&1&1&1&1&1\\
3&7&4&2&2&2&2&2&2&5&10&3&1&2&3&2\\
11&5&10&6&2&2&2&2&8&15&5&7&5&1&1&7\\
8&15&5&7&5&1&1&7&11&5&10&6&2&2&2&2\\
2&5&10&3&1&2&3&2&3&7&4&2&2&2&2&2\\
1&1&1&1&1&1&1&1&1&1&1&1&1&1&1&1
\end{array}
$$
\caption{Fragment of integral closed 2-frieze of width 4}
\label{ex2frieze}
\end{figure}

Figure \ref{ex2frieze} gives an example of integral closed 2-frieze
in which one can easily check that the local rule is satisfied:
$$
 \begin{array}{ccccccc}
 *&B&*\\
 A&E&D\\
 *&C&*
\end{array}\quad
\Longrightarrow
\quad
AD-BC=E.
$$

Propp anounced and conjectured some results on the 2-friezes and
also referred to unpublished work of D. Hickerson.
It seems that nobody had studied this type of frieze in details until \cite{MOT}. 
In \cite{MOT}, the 2-friezes were introduced to study the moduli space of polygons in the projective plane.
It turned out that the space of closed 2-friezes 
of width $n-4$ can be identified to the space of $n$-gons in the projective plane
(provided $n$ is not a multiple of 3).

In the present paper we are interested in combinatorics and algebraic aspects of the 2-friezes.
Our study concerns the particular class of \textit{integral closed 2-friezes}.

The natural question, posed in \cite{Pro} and \cite{MOT}, is;

\begin{ques}\label{classif}
How many integral closed 2-friezes do exist for a given width?
\end{ques}

Let us stress that the answer is known in the case of Coxeter-Conway friezes:
they are counted by Catalan numbers! (This is a consequence of the property (CC2) above.)

In the case of 2-friezes of width $m$ we have the following information
\begin{enumerate}
\item[\textbullet] for $m=1$ and $m=2$, there exist respectively 5 and 51 integral closed 2-friezes; 
this was announced in \cite{Pro} and proved in \cite{MOT},
\item[\textbullet] for $m=3$, there exist at least 868 integral closed 2-friezes; these friezes were found using two independent computer programs (by J. Propp \cite{Pro} and by R. Schwartz used in \cite{MOT}),
\item[\textbullet] for $m>4$, in a private communication V. Fock conjectured to us that there are infinitely many integral closed 2-friezes.
\end{enumerate}

The case $m=4$ still needs to be investigated. The main result of the paper is the following.

\begin{thm}\label{inffrieze}
For any $m>4$ there exist infinitely many integral closed 2-friezes of width $m$.
\end {thm}

Fock's intuition was based on the fact that the closed 2-friezes of width $m$ are related to cluster algebra
associated to the quiver $A_2\times A_m$, which is of infinite type for $m>4$.
Our proof is based on this idea using a procedure to construct integral closed 2-friezes from evaluation of cluster variables. 
However this procedure does not give all the possible friezes. 
Let us mention that our proof of Theorem \ref{inffrieze} uses the positivity conjecture in cluster algebras (see Section \ref{thcla}).

Our next series of results describe operations on the integral 2-friezes and procedures to get new friezes from old ones. 
These operations are given in Theorems \ref{GluThm}, \ref{CutThm}, \ref{Gluxy} in Section \ref{OpSect}.

The paper is organized as follows.
The main sections, Section \ref{ClustSect} and Section \ref{OpSect}, can be read independently.
In Section \ref{ClustSect}, 
we describe the connection between closed 2-friezes and cluster algebras.
The main definitions and results concerning the theory of cluster algebras
that we need are recalled. 
We explain how to get integral friezes from cluster algebras. We finally prove Theorem~\ref{inffrieze} in Section \ref{proofth1}.
In Section \ref{OpSect}, we recall the main properties of the 2-friezes and introduce
a series of algebraic operations on the friezes. 
In particular, we recall the link between closed friezes and moduli spaces of polygons.
This link is helpful to interpret the algebraic operations.
In Section \ref{Conc}, we conclude the paper by refining Question \ref{classif} and mentioning further direction.

\section{Closed 2-friezes and cluster algebra}\label{ClustSect}

Theorem \ref{inffrieze} will be proved with the help of the theory of cluster algebras.
These algebras have been defined by Fomin and Zelevinsky in the early 2000's.
The subject is knowing an exponential growth due to connection to many different fields of mathematics.
In Section \ref{cadef} and \ref{thcla}, we recall the main definitions and results about cluster algebras that will be useful for us. 
All the material can be found in the original work \cite{FZ1}, \cite{FZ2}, \cite{FZ3}, or in \cite[Chapter 3]{GSV}.
We use below the presentation made in \cite{MOT}.

\subsection{Closed 2-frieze}\label{defriz}
A $2$-frieze can be defined as a map $v:(i,j)\mapsto v_{i,j}$ from 
$(\half+\Z)^2\cup\Z^2$ to an arbitrary unital (division) ring, such that 
the following relation holds for all $(i,j) \in (\half+\Z)^2\cup\Z^2$
\begin{equation}\label{friezerule}
\textstyle
v_{i-1,j}\,v_{i,j+1}-\,v_{i,j}\,v_{i-1,j+1}\,=v_{i-\half,j+\half}.
\end{equation}
A $2$-frieze can be pictured as an infinite array as in Figure \ref{vij} below.
\begin{figure}[hbtp]
$$
 \xymatrix{
 &
&\ldots\ar@{-}[rd]
&\ldots \ar@{--}[ld]\ar@{--}[rd]
&\ldots\ar@<2pt>@{-}[ld]
&\\
&\ldots \ar@<2pt>@{-}[rd]
&v_{i-\thalf,j+\half}\ar@{--}[rd]\ar@{--}[ld]
& v_{i-1,j+1}\ar@{-}[ld]\ar@{-}[rd]
&v_{i-\half,j+\thalf}\ar@{--}[rd]\ar@{--}[ld]
&\ar@{-}[ld]\ldots\\ 
&v_{i-\thalf,j-\half}  \ar@{--}[rd]
& v_{i-1,j}\ar@<2pt>@{-}[rd]\ar@<2pt>@{-}[ld]
&v_{i-\half,j+\half}\ar@{--}[ld]\ar@{--}[rd]
& v_{i,j+1}\ar@<2pt>@{-}[ld]\ar@{-}[rd] 
&\ar@{--}[ld] \cdots\\
&v_{i-1,j-1} \ar@{-}[rd]
&v_{i-\half,j-\half} \ar@{--}[rd] \ar@{--}[ld]
&v_{i,j}\ar@{-}[ld]\ar@{-}[rd]
&v_{i+\half,j+\half}\ar@{--}[ld]\ar@{--}[rd]
&v_{i+1,j+1}\ar@{-}[ld]&
 \\
&\ldots &v_{i,j-1}\ar@{-}[rd]
& v_{i+\half,j-\half}\ar@{--}[ld]\ar@{--}[rd]
&v_{i+1,j}  \ar@{-}[ld]&\ldots\\
&&\ldots&\ldots&\ldots&&&&
}
$$
\caption{Indexing the entries of a 2-frieze.}\label{vij}
\end{figure}

A \textit{closed $2$-frieze}, is a map $v:(i,j)\mapsto v_{i,j}$, where
$(i,j)$ is, as before, a pair of integers or of half-integers, restricted to the stripe
$$
-1\leq{}i-j\leq{}m,
$$
where $m$ is a fixed integer called the \textit{width} of the frieze,
and satisfying the local rule \eqref{friezerule} together with the boundary conditions
$v_{i-1,i}=v_{i+\frac{m}{2},i-\frac{m}{2}}=1$ for all $i\in\Z$ or $\Z+\half$.

A closed $2$-frieze is represented by an infinite stripe
\begin{equation}\label{closedfr}
 \begin{matrix}
\cdots& 1&1&1&1&1&\cdots&1&\cdots
 \\[4pt]
\cdots&v_{0,0}&v_{\half,\half}&v_{1,1}&v_{\frac{3}{2},\frac{3}{2}}&v_{2,2}&\cdots&v_{i,i}&\cdots
 \\
&\vdots &\vdots &\vdots &\vdots &\vdots &&\vdots&\\
\cdots&\vdots &\vdots &\vdots &\vdots &\vdots &&v_{i+\frac{m-1}{2},i-\frac{m-1}{2}}&\cdots&\\
 \cdots
& 1&1&1&1&1&\cdots&1&\cdots \\[4pt]
\end{matrix}
\end{equation}

The following statement was known by J.Propp and D. Hickerson and has been written in \cite{MOT}
(it is an analog of the property (CC1) of the Coxeter-Conway friezes mentioned in the introduction).

\begin{prop} \cite{MOT}\label{perio}
In a closed $2$-frieze of width $m$, every row is $2n$-periodic, where $n=m+4$,
i.e. $v_{i+n,j+n}=v_{i,j}$ for all  $(i,j)$.
\end{prop}

\subsection{Formal closed 2-frieze}
A closed $2$-frieze is generically determined by two consecutive columns.
Given $2m$ independent variables $x_1, \ldots, x_{2m}$, the following proposition defines a closed
$2$-frieze with values in the rational fields of fractions $\C(x_1,x_2,\ldots,x_{2m})$
containing the set of variables  $x_1, \ldots, x_{2m}$ into two consecutive columns.

\begin{prop}\cite{MOT}\label{formal}
There exists a unique closed $2$-frieze of width $m$, with values in the rational fields of fractions
$\C(x_1,x_2,\ldots,x_{2m})$ containing the following sequences
\begin{equation}
\label{ClustF}
 \begin{array}{cccccc}
\cdots&1& 1&1&1&\cdots
 \\[4pt]
&&x_1&x_{m+1}&&
 \\[4pt]
&&x_{m+2}&x_2&&
 \\[4pt]
&&x_3&x_{m+3}&&
 \\[4pt]
&&\vdots&\vdots&&
 \\
&&\vdots &\vdots &&\\
\cdots&1& 1&1&1&\cdots
\end{array}
\end{equation}
where $x_1$ is in position $v_{0,0}$.
Furthermore, all the entries of the $2$-frieze are Laurent polynomials in $x_1, \ldots, x_{2m}$.
\end{prop}

The formal $2$-frieze characterized in the above Proposition is denoted by $F(x_1, \ldots, x_{2m})$.

\begin{ex}
\label{ClustFive}
Case $m=1$
$$
\begin{array}{cccccccccccccccc}
\cdots &1 & 1 & 1 & 1& 1& 1& 1& \cdots\\[6pt]
\cdots & x_1 & x_2&   \frac{x_2+1}{x_1}&  \frac{x_1+x_2+1}{x_1x_2}& \frac{x_1+1}{x_2}&x_1 & x_2&\cdots \\[6pt]
\cdots &1& 1& 1& 1& 1& 1& 1& \cdots \\
\end{array}
$$
\end{ex}

\begin{ex}
\label{ClustSix}
Case $m=2$
\begin{equation*}
\begin{array}{cccccccccccccccc}
\cdots &1 & 1 & 1 & 1& 1& 1&  1& 1& \cdots\\[6pt]
\cdots & x_1 & x_3&   \frac{x_3+x_2}{x_1}& \frac{(x_3+x_2)(x_4+x_1)}{x_1x_3x_4}
& \frac{(x_1+x_4)(x_2+x_3)}{x_2x_3x_4} & \frac{x_1+x_4}{x_2}&x_4 & x_2&\cdots \\[6pt]
\cdots & x_4 & x_2&   \frac{x_2+x_3}{x_4}& \frac{(x_2+x_3)(x_1+x_4)}{x_4x_2x_1}
& \frac{(x_4+x_1)(x_3+x_2)}{x_3x_1x_2} & \frac{x_4+x_1}{x_3}&x_1 & x_3&\cdots \\[6pt]
\cdots &1& 1& 1& 1& 1& 1& 1& 1& \cdots \\
\end{array}
\end{equation*}
\end{ex}

\begin{rem}
The Laurent phenomenon described in Proposition \ref{formal} was mentioned in \cite{Pro} and proved in \cite{MOT} using a link to cluster algebras
(this link also implies the periodicity described in Proposition \ref{perio} but periodicity has been established by elementary method in~\cite{MOT}).
As an easy consequence of the Laurent phenomenon one can obtain integral closed 2-friezes by setting the inital
variables $x_1, \ldots, x_{2m}$ to be equal to 1. The link to cluster algebras will actually provide more information. 
\end{rem}

\subsection{Cluster algebras: basic definitions}\label{cadef}
A cluster algebra is a commutative associative algebra.
This is a subalgebra of a field of rational fractions
in $N$ variables.
The cluster algebra is presented by generators and relations. 
The generators are collected in packages called \textit{clusters} of fixed cardinality $N$. 
The constant $N$ is called the rank of the algebra.
The generators and relations are not given from the beginning.
They are obtained recursively using a combinatorial procedure encoded in a matrix, or an oriented graph
with no loops and no $2$-cycles.

We give here an explicit construction of the (complex or real) cluster algebra  $\A(\Qc)$
starting from a finite oriented  connected graph $\Qc$ with no loops and no $2$-cycles
(there exists more general construction of cluster algebras but the one given here is enough for our purpose).
Let $N$ be the number of vertices of $\Qc$,
the set of vertices is then identified with the set $\{1, \ldots, N\}$.
The algebra $\A(\Qc)$ is a subalgebra of the field of fractions $\C(x_1,\ldots, x_N)$ in $N$  
variables $x_1,\ldots, x_N$ (or over $\R$, in the real case). 
The generators and relations of $\A(\Qc)$ are given using a recursive procedure called
\textit{seed mutations} that we describe below.

A \textit{seed} is a couple 
$$
\Sigma=\left((t_1, \ldots, t_N) , \;\Rc\right),
$$
where $\Rc$ is an arbitrary finite oriented graph with $N$ vertices
and where $t_1, \ldots, t_N$ are free generators of $\C(x_1,\ldots, x_N)$ labeled by the vertices of the graph $\Rc$. 
The \textit{mutation at  vertex}  $k$ of the seed 
$\Sigma$ is a new seed $\mu_k(\Sigma )$ defined by
\begin{enumerate}
\item[\textbullet]
$\mu_k(t_1, \ldots, t_N)=(t_1, \ldots, t_{k-1},t'_k, t_{k+1},\ldots, t_N)$ where
\begin{equation}\label{exrel}
\displaystyle
t'_k=\dfrac{1}{t_k}\left(\prod\limits_{\substack{\text{arrows in }\Rc\\ i\rightarrow k }}\; t_i
 \quad+\quad 
\prod\limits_{\substack{\text{arrows in }\Rc\\ i\leftarrow k }}\;t_i\right)
\end{equation}
\item[\textbullet] 
$\mu_k(\Rc)$ is the graph obtained from $\Rc$ by applying the following transformations
\begin{enumerate}
\item for each possible path $i\rightarrow k \rightarrow j$ in $\Rc$, add an arrow $i\rightarrow j$,
\item reverse all the arrows leaving or arriving at $k$,
\item remove a maximal collection of 2-cycles, 
\end{enumerate}
\end{enumerate}
(see Example \ref{exmut}  below for a seed mutation).

Starting from the initial seed $\Sigma_0=((x_1,\ldots, x_N), \Qc)$, one produces  $N$ new seeds 
$\mu_k(\Sigma_0)$, $k=1,\ldots, N$. 
Then one applies all the possible mutations to all of the created new seeds, and so on.
The set of rational functions appearing in any of the seeds produced during the mutation process
is called a \textit{cluster}. 
The functions in a cluster are called \textit{cluster variables}.
The cluster algebra $\A(\Qc)$ is the subalgebra of  $\C(x_1,\ldots, x_N)$ generated by all the cluster variables.

\begin{ex}\label{exmut}
In the case $n=4$, consider the seed $\Sigma=$
$$
(t_1,t_2,t_3,t_4), \quad
\Rc=
\xymatrix{
1\ar@{->}[r]
&2\ar@{->}[d]
\\
3\ar@{<-}[r]\ar@{->}[u]
&4
}.$$
The mutation at vertex 1 gives 
$$
\mu_1(t_1,t_2,t_3,t_4)=\Big(\frac{t_2+t_3}{t_1},t_2,t_3,t_4\Big), \quad
\mu_1(\Rc)=\quad
\xymatrix{
1\ar@{<-}[r]
&2\ar@{->}[d]\ar@{<-}[ld]
\\
3\ar@{<-}[r]\ar@{<-}[u]
&4}.
$$
Performing the mutation $\mu_2$  on $\mu_1(\Rc)$ leads to the following graph
$$
\mu_2\mu_1(\Rc)=\quad
\xymatrix{
1\ar@{->}[r]
&2\ar@{<-}[d]\ar@{->}[ld]
\\
3
&4}.
$$
The underlying non-oriented graph of $\mu_2\mu_1(\Rc)$ is  the Dynkin diagram of type $D_4$.
The algebra $A(\Rc)$ is referred to as the cluster algebra of type $D_4$ in the terminology of \cite{FZ2}.
It is known that in this case the mutation process is finite, 
meaning that applying all the possible mutations to all the seeds leads to a finite number of seeds and therefore to a finite number (24) of cluster variables.
\end{ex}

\subsection{Cluster algebras: fundamental results}\label{thcla}

To prove Theorem \ref{inffrieze} we will need the following fundamental theorems
on cluster algebras.

The first result relates the classification of cluster algebras to that of Lie algebras,
using Dynkin graphs which are any orientations of Dynkin diagrams.

\begin{theo}[Classification \cite{FZ2}]
The cluster algebra $\A(\Qc)$ has finitely many cluster variables if and only if the initial graph $\Qc$ is mutation-equivalent to a Dynkin graph of type $A,D,E$.
\end{theo}

More general definition of a cluster algebra may allow to include all the Dynkin types in the classification. We do not need such general construction.

The second result establishes a surprising phenomenon of simplification in the expressions of cluster variables.
\begin{theo}[Laurent Phenomenon \cite{FZ1}]
Every cluster variable can be expressed as a Laurent polynomial with integer coefficients in the variables of any given cluster.
\end{theo}

\begin{conj}[Positivity \cite{FZ1}]
The Laurent polynomials in the above theorem have positive integer coefficients.
\end{conj}

The positivity conjecture has been proved in several cases. 
In particular, it has been proved\footnote{I am grateful to the anonymous referee for providing me with the reference.} by Nakajima \cite{Nak} in the case where the graph is bipartite.
We will be interested in the cluster algebra associated to the quiver \eqref{Graph} (see below)
which is a bipartite graph. So, in our case, the positivity conjecture is a theorem; we will need this statement in the proof of Lemma \ref{infclasse}.

\subsection{Algebra of functions on the closed 2-frieze }
We denote by $\A_m$ the subalgebra of $\C(x_1, \ldots, x_{2m})$ generated by the entries of the frieze $F(x_1, \ldots, x_{2m})$, defined by \eqref{ClustF}.

\begin{theo}[\cite{MOT}]
The algebra $\A_m$ associated with the 2-frieze $F(x_1, \ldots, x_{2m})$ is a subalgebra of the cluster algebra $\A( \Qc_m)$, where
$\Qc_m$ is the following oriented graph
\begin{equation}
\label{Graph}
\xymatrix{
&1\ar@{->}[r]
&2\ar@{<-}[r]\ar@{->}[d]
&3\ar@{->}[r]
&\cdots
&\cdots \ar@{<-}[r]
&m-1\ar@{->}[r]
&m\ar@{->}[d]\\
&
m+1\ar@{<-}[r]\ar@{->}[u]
&m+2\ar@{->}[r]
&m+3\ar@{->}[u]\ar@{<-}[r]
&\cdots&\cdots \ar@{->}[r]
&2m-1\ar@{->}[u]\ar@{<-}[r]
&2m\\
}
\end{equation}
Moreover, the set of variables contained in two consecutive columns of the 2-frieze
$F(x_1, \ldots, x_{2m})$ is a cluster in $\A( \Qc_m)$.
\end{theo}

Note that the orientation of the last square in \eqref{Graph} depends on the parity of $m$.
Note also that  $\Qc_m$ is the cartesian product of two Dynkin graphs: $\Qc_m=A_2\times A_m$.

\begin{rem}
It is well known that the graph $\Qc_m$ is mutation equivalent to the Dynkin graph of type $D_4, E_6, E_8$
for $m=2,3,4$ respectively. For $m\geq 5$, the graph  $\Qc_m$ is not mutation equivalent to any Dynkin graph.
Therefore, the number of cluster variables in $\A( \Qc_m)$ is infinite for $m\geq{}5$.
\end{rem}

\begin{rem}\label{bipart} (i) The algebra $\A_m$ is related to what Fomin and Zelevinsky \cite{FZ4} called the bipartite belt.
Indeed, the graph $\Qc_m$ is bipartite, i.e. 
one can associate a sign $\e(i)=\pm$ to each vertex $i$
of the graph so that any two connected vertices in
$\Qc_m$ have different signs. Let us assume that $\e(1)=+$ 
(this determines automatically all the signs of the vertices).

Consider the iterated mutations
$$
\mu_+=\prod_{i: \e(i)=+}\;\mu_i, \qquad 
\mu_-=\prod_{i:\e(i)=-}\;\mu_{i}.
$$
Note that $\mu_i$ with $\e(i)$ fixed commute with each other, and therefore $\mu_+$ and $\mu_-$ are involutions.

One can easily check that
the result of the mutation of the graph \eqref{Graph} by $\mu_+$ and $\mu_-$
is the same graph with reversed orientation:
$$
\mu_+(\Qc_m)=\Qc_m^{\hbox{op}},
\qquad
\mu_-(\Qc_m^{\hbox{op}})=\Qc_m.
$$

Consider the seeds of $\A(\Qc_m)$
obtained from $\Sigma_0$ by applying successively $\mu_+$ or $\mu_{-}$:
\begin{equation*}
\Sigma_0,\quad
\mu_+(\Sigma_0),\quad \mu_-\mu_+(\Sigma_0),\quad \ldots, \quad
\mu_{\pm}\mu_{\mp}\cdots \mu_-\mu_+(\Sigma_0),\quad \ldots
\end{equation*}
The cluster variables in each of the above seeds correspond
precisely to two consecutive columns in the 2-frieze pattern \eqref{ClustF}.
This set of seeds is called the bipartite belt of $\A(\Qc_m)$, see \cite{FZ4}.

(ii)
Entries of a 2-frieze $F(x_1, \ldots, x_{2m})$ are the cluster variables in $\A(\Qc_m)$
that can be obtained from the initial seed by applying sequences of $\mu_+$ and $\mu_{-}$.
It is not known how to characterize the cluster variables of $\A(\Qc_m)$ that do not appear in the 2-frieze.

\end{rem}

\subsection{Counting integral 2-friezes}
We are now interested in the integral closed 2-friezes, i.e. 
arrays as \eqref{closedfr} in which the entries $v_{i,j}$ are positive integers.
Due to periodicity, see Proposition~\ref{perio}, we represent a closed 2-frieze by a fundamental fragment of size $2n\times(n-4)$ (we will always choose the fragment whose entries in the first row are $v_{0,0},\ldots, v_{n-\half, n-\half}$).
Repeating the same fragment infinitely many times on left and right of the initial one, will lead to the complete infinite 2-frieze.

It is clear that two different fragments can produce to the same infinite frieze. 
For instance, permuting cyclically the columns of a fragment gives another fragment that produces the same infinite frieze. 
Also, rewriting a fragment from right to left may lead to another well-defined infinite frieze.

Define the following two operations on a fragment
\begin{equation}\label{tausig}
\begin{array}{lcl}
\tau \;\cdot \;
 \begin{matrix}
 1&1&\cdots&1&1 \\
 a_1&a_2&\cdots&a_{2n-1}&a_{2n}\\ 
  b_1&b_2&\cdots&b_{2n-1}&b_{2n}\\ 
  \vdots&  \vdots&  &\vdots&  \vdots\\
  1&1&\cdots&1&1 \\[4pt]
\end{matrix}
 &=&
 \begin{matrix}
 1&1&\cdots&1&1 \\
 a_2&a_3&\cdots&a_{2n}&a_{1}\\ 
  b_2&b_3&\cdots&b_{2n}&b_{1}\\ 
  \vdots&  \vdots&  &\vdots&  \vdots\\
  1&1&\cdots&1&1 \\[4pt]
\end{matrix}
\\[40pt]
\sigma\; \cdot
 \begin{matrix}
 1&1&\cdots&1&1 \\
 a_1&a_2&\cdots&a_{2n-1}&a_{2n}\\ 
  b_1&b_2&\cdots&b_{2n-1}&b_{2n}\\ 
  \vdots&  \vdots&  &\vdots&  \vdots\\
  1&1&\cdots&1&1 \\[4pt]
\end{matrix}
&=&
 \begin{matrix}
 1&1&\cdots&1&1 \\
 a_{1}&a_{2n}&\cdots&a_3&a_2\\ 
b_{1}  &b_{2n}&\cdots&b_3&b_2\\ 
  \vdots&  \vdots&  &\vdots&  \vdots\\
  1&1&\cdots&1&1 \\[4pt]
\end{matrix}
\\
\end{array}
\end{equation}
Using the definition of closed 2-friezes as maps $v:(i,j)\mapsto{}v_{i,j}$, one
has
$$
\tau\cdot{}v:(i,j)\mapsto{}v_{i+\half,j+\half},
\qquad
\sigma\cdot{}v:(i,j)\mapsto{}v_{-j,-i}\,.
$$

\begin{prop}
The operations $\tau$ and $\sigma$ generate an action of the dihedral group of order $4n$ on 
the set of fragments of integral closed 2-friezes of width $n-4$.
\end{prop}

\begin{proof}
One checks the relations $\s\t\s=\t^{-1}$ and $\s^2=\t^{2n}=\id$.
\end{proof}

In the sequel we are interested in the problem of counting fragments of
integral closed 2-friezes of a given width.

\begin{ex}
\label{notsurj}
It was proved in \cite{MOT} that the following five fragments of 2-friezes produce
all the integral 2-friezes of width 2 (modulo the action of the dihedral group).
\begin{equation}
\label{SixThree}
\begin{array}{rrrrrrrrrrrrr}
1&1&1&1&1&1&1&1&1&1&1&1\\
1&1&2&4&4&2&1&1&2&4&4&2\\
1&1&2&4&4&2&1&1&2&4&4&2\\
1&1&1&1&1&1&1&1&1&1&1&1
\end{array}
\end{equation}
\begin{equation}
\label{SixFour}
\begin{array}{rrrrrrrrrrrrr}
1&1&1&1&1&1&1&1&1&1&1&1\\
1&1&3&6&3&1&1&2&3&3&3&2\\
1&2&3&3&3&2&1&1&3&6&3&1\\
1&1&1&1&1&1&1&1&1&1&1&1
\end{array}
\end{equation}
\begin{equation}
\label{SixFive}
\begin{array}{rrrrrrrrrrrrr}
1&1&1&1&1&1&1&1&1&1&1&1\\
1&1&4&6&2&1 & 2&3&2&2&4&3 \\
2&3&2&2&4&3 &1&1&4&6&2&1\\
1&1&1&1&1&1&1&1&1&1&1&1
\end{array}
\end{equation}
\begin{equation}
\label{SixTwo}
\begin{array}{rrrrrrrrrrrrr}
1&1&1&1&1&1&1&1&1&1&1&1\\
1&3&5&2&1&3&5&2&1&3&5&2\\
5&2&1&3&5&2&1&3&5&2&1&3\\
1&1&1&1&1&1&1&1&1&1&1&1
\end{array}
\end{equation}
\begin{equation}
\begin{array}{rrrrrrrrrrrrr}
1&1&1&1&1&1&1&1&1&1&1&1\\
2&2&2&2&2&2&2&2&2&2&2&2\\
2&2&2&2&2&2&2&2&2&2&2&2\\
1&1&1&1&1&1&1&1&1&1&1&1
\end{array}\label{SixOne}
\end{equation}
The action of the dihedral group on the above fragments
leads to 51 different fragments:
the orbits of the fragments \eqref{SixThree}-\eqref{SixOne} 
contain  6, 12, 24, 8 and 1 elements, respectively.
\end{ex}

\subsection{Integral 2-friezes as evaluation of cluster variables}\label{defev}
The easiest way to obtain a closed integral 2-frieze is to make an evaluation 
of the formal 2-frieze $F( x_1, \ldots, x_{2m})$
by setting all the initial variables $x_i=1$. 
All the entries in the resulting frieze will be positive integers.
Indeed, the Laurent phenomenon ensures that the entries are well-defined  and integers, 
and the local rule \eqref{friezerule} ensures that
the entries are positive (here we do not need the positivity conjecture).

The above idea can be extended by setting all the variables in an arbitrary cluster to be equal to~1.
Given an arbitrary cluster $\bc=(c_1,\ldots,c_{2m})$ in  $\A(\Qc_m)$,
every entry in the  2-frieze $F( x_1, \ldots, x_{2m})$
can be expressed as a Laurent polynomial in $c_i$. 
This gives a new formal 2-frieze $F(\bx(\bc))$.
Setting $c_i=1$ for all $i$, all the entries of $F(\bx(\bc))$ become positive integers.
Indeed, they are integers since their expressions are Laurent polynomials in $\bc$, and they are positive because 
the expressions are obtained from $\bc$ using a sequence of exchange relations that are subtraction free.
This procedure defines a map
$$
\begin{array}{rcl}
\ev: \{ \text{cluster of } \A(\Qc_m)\} &\rightarrow& \{\text{fragment of integral 2-frieze of width }m\}\\[6pt]
\bc &\mapsto & F(\bx(\bc))\left|_{\bc=(1,\ldots,1)}\right..\\[4pt]
\end{array}
$$
where the fragment representing the frieze is chosen starting by the two columns containing the values of $(x_1,x_2,\ldots)$.
\begin{defn}
Integral friezes produced by the map  $\ev$ are called \textit{unitary friezes}.
\end{defn}

\begin{rem}
The map $\ev$ is not necessarily surjective. For instance, in the case of $2$-friezes of width 2,
 there are exactly 51 fragments, see Example \ref{notsurj}, 
 but the corresponding cluster algebra is of Dynkin type $D_4$ which is known to have 50 clusters. 
 One therefore deduces that at least one fragment is not a unitary frieze.
\end{rem}

\begin{lem}\label{unitdihed}
If a fragment of $2$-frieze is in the image of $\ev$, then all the fragments obtained under the action of the dihedral group are also in the image of $\ev$.
\end{lem}

\begin{proof} Let us consider a fragment $\ev(\bc)$, 
for some cluster $\bc=\mu_{i_1}\cdots \mu_{i_k}(\bx)$.
To describe the action $\s$  we introduce the sequence of mutations
 $\mu_+:=\mu_{m+1}\mu_2\mu_{m+3}\mu_4\cdots$ (note that $\mu_+^2=\id$, see Remark \ref{bipart}).
One has
$$
 \sigma\cdot \ev(\bc)= \ev(\mu_{{i_1}}\cdots \mu_{{i_k}}\mu_+\bx).
$$
The action that gives the fragment with the first two columns in the reverse order, 
 i.e. the action of $\sigma\tau $, is easily obtained by reversing the roles of the indices $i\leftrightarrow i+m$. Hence,
$$
 \sigma\tau\cdot \ev(\bc)= \ev(\mu_{\overline{i_1}}\cdots \mu_{\overline{i_k}}\bx).
$$
where we use the notation $\bar{i}:=i+m \mod 2m$.
One finally deduces 
$$
\tau\cdot \ev(\bc)= \ev(\mu_{\overline{i_1}}\cdots \mu_{\overline{i_k}}\mu_{+}\bx).
$$

\end{proof}

\begin{prop}\label{unitfrieze}
The integral 2-friezes \eqref{SixThree}-\eqref{SixTwo} are unitary friezes. 
Under the action of the dihedral group the friezes \eqref{SixThree}-\eqref{SixTwo}  
produce 50 different friezes coming for the evaluation of the 50 clusters of the cluster algebra
$\A_{\Qc_2}\simeq \A(D_4)$. Consequently, \eqref{SixOne} is not a unitary frieze.
\end{prop}

\begin{proof}
The formal frieze of width 2 is displayed in Example \ref{ClustSix}.
We construct the fragments of frieze \eqref{SixThree}-\eqref{SixTwo}
as image of the map $\ev$ (with the convention that the two first columns of the fragment contain the initial cluster variables
$(x_1,\ldots, x_{4})$).

The fragment of frieze \eqref{SixThree} is realized as $\ev(\bx)$ where 
$\bx=(x_1,\ldots, x_{4})$ is the initial cluster.

The fragment of frieze \eqref{SixFour}, is easily obtained as $\ev(\bc)$ where 
$\bc=\mu_2(\bx)$.

The fragment of frieze \eqref{SixFive} is realized as $\ev(\bd)$ where 
$\bd$ is obtained from the initial cluster by performing the mutations $\mu_2\mu_4$.
Indeed, one can check this sequence of mutations transforms the initial seed $(\bx,\Qc_2)$ into
$$
\xymatrix{&
&1\ar@{<-}[rd]
&2\ar@{->}[d]
\\
&
(d_1,d_2,d_3, d_{4})\;, &3\ar@{->}[r]
&4
}
$$
where
$$
\left\{
\begin{array}{rcl}
d_1&=&x_1,\\[10pt]
d_2&=&\dfrac{x_1+x_4}{x_2},\\[10pt]
d_3&=&x_3,\\[10pt]
d_4&=&\dfrac{x_1x_2+x_1x_3+x_3x_4}{x_2x_4}.
\end{array}
\right.
$$
One then sees
$$
(d_1,d_2,d_3, d_{4})=(1,1,1,1)\Longleftrightarrow (x_1,x_2,x_3, x_{4})=(1,3,1,2).
$$

The  fragment of frieze \eqref{SixTwo} is realized as $\ev(\be)$ where 
$\be$ is obtained from the initial cluster by applying the mutations 
$\mu_4\mu_2\mu_3\mu_4\mu_2$.
Indeed, one can check this sequence of mutations transforms the initial seed $(\bx,\Qc_2)$ into
$$
\xymatrix{&
&1\ar@{->}[rd]
&2\ar@{->}[d]
\\
&
(e_1,e_2,e_3, e_{4})\;, &3\ar@{->}[r]
&4
}
$$
where
$$
\left\{
\begin{array}{rcl}
e_1&=&x_1,\\[10pt]
e_2&=&\dfrac{x_2+x_3}{x_4},\\[10pt]
e_3&=&\dfrac{x_1x_2+x_1x_3+x_2x_4+x_3x_4}{x_2x_3x_4},\\[10pt]
e_4&=&\dfrac{x_1x_2+x_1x_3+x_2x_4}{x_3x_4}.
\end{array}
\right.
$$
One can then check that 
$$
(e_1,e_2,e_3, e_{4})=(1,1,1,1)\Longleftrightarrow (x_1,x_2,x_3, x_{4})=(1,2,3,5).
$$

By Lemma \ref{unitdihed}, one deduces that the 50 fragments produced by  \eqref{SixThree}-\eqref{SixTwo} under the action of the dihedral group are all realized as evaluation of a cluster at $(1,\cdots,1)$. Since in this case one has exactly 50 clusters, the remaining fragment \eqref{SixOne} can not be realized this way. 
This can also be checked by hand. Evaluating the initial cluster at $(2,2,2,2)$ and performing a sequence of mutations in all the possible directions will always lead to two different sets  $\{2,2,2,2\}$ and $\{2,2,2,3\}$ for the values of the cluster variables.
\end{proof}

\begin{rem}\label{width3}
For the 2-friezes of width $m=3$ the associated cluster algebra is of type $A_2\times A_3\simeq E_6$.
In this type there are 833 clusters, but we found already 868 friezes, so that we can expect at least 35 non-unitary friezes.
\end{rem}

\subsection{Proof of Theorem \ref{inffrieze}}\label{proofth1}
In this section we fix $m>4$.
We denote by $\cC$ the set of all clusters in the cluster algebra $\A_{\Qc_m}$.
It is known that $\cC$ is infinite.
We define the following relation on $\cC$:
$$
\bc\sim \bd \text{ if and only if } c_1=\cdots=c_{2m}=1 \text { implies } d_1=\cdots=d_{2m}=1,
$$
where the $c_i$ and $d_i$ are the variables in the clusters $\bc$ and $\bd$ respectively.

\begin{lem}\label{Pfequiv}
The relation $\sim$ is an equivalence relation on $\cC$.
\end{lem}

\begin{proof}
It is not clear from the definition that $\sim$ is symmetric.
Given two clusters $\bc$ and $\bd$ there exist two  $2m$-tuples of Laurent polynomials in $2m$ variables
$P_{\bc,\bd}=(P_1,\cdots, P_{2m})$ and $P_{\bd,\bc}=(Q_1,\cdots, Q_{2m})$ such that 
$$
\bc=P_{\bc,\bd}(\bd)=(P_1(\bd),\cdots, P_{2m}(\bd))\; \text{ and } \;\bd=P_{\bd,\bc}(\bc)=(Q_1(\bc),\cdots, Q_{2m}(\bc)).
$$
The transition maps $P_{\bc,\bd}$ and $P_{\bd,\bc}$ define two bijections, inverse one to the other, from $(\C^*)^n$ to $(\C^*)^n$. Therefore,
$$
(1,\ldots,1)=P_{\bc,\bd}(1,\ldots,1)\Longleftrightarrow (1,\ldots,1)=P_{\bd,\bc}(1,\ldots,1).
$$
This proves the symmetry of the relation.
\end{proof}

We denote by $\cC/\!\sim$ the set of all equivalence classes, and we denote by $\bar{\bc}$ the equivalence 
class of an element $\bc \in \cC$.

\begin{lem}\label{bev}
The following map is well defined and is injective
$$
\begin{array}{cccl}
\bev: &\cC/\!\sim&\longrightarrow& \{\text{fragments of integral 2-friezes of width }m\}\\[6pt]
&\bar{\bc}&\mapsto& \ev(\bc)
\end{array}
$$
where $\ev$ is the function introduced in Section \ref{defev}.
\end{lem}

\begin{proof}
We need to prove that $\ev(\bc)=\ev(\bd)$ if and only if $\bc\sim \bd$.
This can be done using as in the proof of Lemma \ref{Pfequiv}
the transition functions $P_{\bc,\bd}$,  $P_{\bc,\bx}$,  $P_{\bx,\bd}$ between the different clusters  and the reciprocals.
\begin{eqnarray*}
\ev(\bc)=\ev(\bd)
&\Longleftrightarrow&P_{\bx,\bc}(1,\ldots,1)=P_{\bx,\bd}(1,\ldots,1)\\
&\Longleftrightarrow&(1,\ldots,1)=P_{\bc,\bx}P_{\bx,\bd}(1,\ldots,1)\\
&\Longleftrightarrow&(1,\ldots,1)=P_{\bc,\bd}(1,\ldots,1)\\
&\Longleftrightarrow&\bc\sim\bd
\end{eqnarray*}
\end{proof}

To complete the proof of Theorem \ref{inffrieze}, we show that the set $\cC/\!\sim$ is infinite.

\begin{lem}\label{infclasse}
If $\bc\sim\bd$ then $\bc=(c_1,\ldots,c_{2m})$ is a permutation of $\bd=(d_1,\ldots,d_{2m})$.
\end{lem}

\begin{proof}
Given two clusters $\bc=(c_1,\ldots,c_{2m})$ and $\bd=(d_1,\ldots,d_{2m})$, 
one can express every variable in one of the clusters as Laurent polynomial with positive integer coefficients in the variables in the other cluster (thanks to Nakajima's results \cite{Nak}).
The equivalence $\bc\sim \bd$ implies that  the expressions are actually unitary Laurent monomials.
Write for instance $c_i=\Pi_{1\leq j\leq 2m}d_j^{k_j}$ with $k_j\in \Z$.
 If the exponent $k_{j}$ is negative, then the expansion of $c_i$ in the variables of $\mu_j(\bd)$ is not a Laurent polynomial. 
Therefore, one deduces that the expressions of the $c_i$'s are just monomials (not Laurent) in the variables of $\bd$, and by symmetry, the $d_i$'s are also monomials in the variables of $\bc$.
This happens if and only if the set of variables in $\bc$ is the same as the set of variables in~$\bd$.
\end{proof}

By Lemma \ref{infclasse}, there is a finite number of clusters in a given equivalence class $\bar{\bc}$
of $\cC/\sim$. 
Since $\cC$ is infinite, one deduces there are infinitely many classes in $\cC/\sim$.
The injective map $\bev$ in Lemma~\ref{bev} produces infinitely many integral 2-friezes. 

Theorem \ref{inffrieze} is proved. 

\section{Cutting and gluing $2$-friezes}\label{OpSect}

\subsection{The algebraic operations}

The first operation that we describe on the $2$-friezes, is equivalent to the connected sum defined in \cite{MOT}. 
It produces a new integral closed 2-frieze starting from two smaller 2-friezes.

\begin{thm}
\label{GluThm}
Given two integral closed 2-friezes, of width $m$ and $\ell$, 
the following gluing of two columns on the top of the other over the pair $1\;1$ 
\begin{equation}\label{glucol}
\begin{array}{rccccl}
\cdots&1 & 1 & 1 &1&\cdots\\
 &&\circ&\circ&&\\
& &\vdots  & \vdots & &\\
& &\vdots  & \vdots & &\\
& & \circ & \circ & &\\
\cdots&1 & \mathbf{1} & \mathbf{1} &1& \cdots
\end{array},
\begin{array}{rccccl}
\cdots&1 & \mathbf{1} & \mathbf{1} &1&\cdots\\
 &&\bullet &\bullet&&\\
& &\vdots  & \vdots & &\\
& & \bullet & \bullet& &\\
\cdots&1 & 1 & 1 &1& \cdots
\end{array}
\qquad\longmapsto
\begin{array}{rccccl}
\cdots&1 & 1 & 1 &1&\cdots\\
 &&\circ&\circ&&\\
& &\vdots  & \vdots & &\\
& &\vdots  & \vdots & &\\
& & \circ & \circ & &\\
&  & \mathbf{1} & \mathbf{1} &&\\
 &&\bullet &\bullet&&\\
& &\vdots  & \vdots & &\\
& & \bullet & \bullet& &\\
\cdots&1 & 1 & 1 &1& \cdots
\end{array}
\end{equation}
leads to a new integral closed 2-frieze of width $m+\ell+1$.
\end{thm}

The second operation breaks a 2-frieze into a smaller 2-frieze.
\begin{thm}
\label{CutThm}
Cutting above a pair $x,y$ in an integral 2-frieze
\begin{equation}
\label{BigPat}
\begin{array}{rccl}
\cdots & 1 & 1 & \cdots\\
 &\vdots &\vdots&\\
& x & y &\\
 & u & v &\\
&\vdots&\vdots&\\
\cdots & 1 & 1 & \cdots
\end{array}
\longmapsto
\begin{array}{rccl}
\cdots & 1 & 1 & \cdots\\
 &\vdots&\vdots &\\
& x & y &\\
\cdots & 1 & 1 & \cdots
\end{array}
\end{equation}
gives a new integral 2-frieze if and only if
\begin{equation}
\label{Cond}
u\equiv 1 \mod y,
\qquad
v\equiv 1\mod x.
\qquad
\end{equation}
\end{thm}

The last operation glues 2-friezes in a more general way than the operation of Theorem \ref{GluThm}.
\begin{thm}
\label{Gluxy}
Gluing two integral friezes, of width $m$ and $\ell$,  over a pair $x,y$, as follows, 
\begin{equation}
\label{glufrixy}
\begin{array}{rccl}
\cdots & 1 & 1 & \cdots\\
 &\vdots &\vdots&\\
& r& s &\\
 & x & y &\\
\cdots & 1 & 1 & \cdots
\end{array},
\begin{array}{rccl}
\cdots & 1 & 1 & \cdots\\
& x & y &\\
 & u & v &\\
&\vdots&\vdots&\\
\cdots & 1 & 1 & \cdots
\end{array}
\longmapsto
\begin{array}{rccl}
\cdots & 1 & 1 & \cdots\\
 &\vdots &\vdots&\\
& r& s &\\
 & x & y &\\
 & u & v &\\
&\vdots&\vdots&\\
\cdots & 1 & 1 & \cdots
\end{array}
\end{equation}
gives a new 2-frieze of width $m+\ell-1$ if and only if
\begin{equation}
\label{Cond2}
u\equiv r\equiv 1 \mod y,
\qquad
v\equiv s\equiv 1\mod x.
\qquad
\end{equation}
\end{thm}

Let us mention that conditions \eqref{Cond} and \eqref{Cond2} hold true for 
$(x,y)=(1,1), (2,1)$ and $(1,2)$ independently of the values of $u,v$, because of the local rule. 
This allows us to cut or glue a frieze whenever one of these pairs appears in the pattern.

\subsection{Entries in the 2-friezes}
Recall that the entries of a 2-frieze
are given by two sequences
$(v_{i,j})_{(i,j)\in \Z^2}$
and $(v_{i+\half,j+\half})_{(i,j)\in \Z^2}$ (see Figure \ref{vij} in Section \ref{defriz}).
In the geometric situation it will be useful to complete the closed 2-friezes by two rows of 0's 
above and under the frieze.

$$
 \begin{matrix}
 \cdots
&0& 0&0&0&\cdots
 \\
\cdots&0& 0&0&0&\cdots
 \\
\cdots& 1&1&1&1&\cdots
 \\
\ldots&v_{i-\half,i-\half}&v_{i,i}&v_{i+\half,i+\half}& v_{i+1,i+1}&\ldots
 \\[4pt]
 &v_{i,i-1}&v_{i+\half,i-\half}&v_{i+1,i}&v_{i+\thalf,i+\half}&\\
&\vdots &\vdots &\vdots &\vdots &&\\
 \cdots
& 1&1&1&1&\cdots \\
\cdots& 0&0&0&0&\cdots
 \\
\cdots& 0&0&0&0&\cdots
\end{matrix}
$$

\subsection{2-friezes and moduli space of polygons}

The space of all 2-friezes (with complex coefficients) is an interesting algebraic variety
closely related to the famous moduli space $\cM_{0,n}$ of genus-zero algebraic curves
with $n$ marked points.
This geometric interpretation will be useful in the sequel.

We call \textit{$n$-gon} a cyclically ordered $n$-tuple of points
$$
V_1,\ldots,V_n\in\C^3,
$$
i.e., we assume $V_{i+n}=V_i$, such that any three consecutive points obey the
normalization condition
\begin{equation}
\label{Normal}
\det(V_{i}, V_{i+1}, V_{i+2})=1.
\end{equation}
Following \cite{MOT}, we consider space of $n$-gons in $\C^3$ modulo the action of
the Lie group $\SL(3,\C)$, via
$$
\cC_n:=\lbrace n\hbox{-gons}\rbrace/\SL_3(3,\C).
$$
This space has a natural structure of algebraic variety.

\begin{prop}\cite{MOT} 
\label{Isop}
The set of closed $2$-friezes over $\C$ of width $m$ is in bijection with $\cC_{m-4}$.
\end{prop}

Starting from a closed 2-frieze one can construct the $n$-gon $(V_i)$ from any three consecutive diagonals, 
for instance:
$$
V_1=\left(
\begin{array}{l}
v_{1,j-2}\\[4pt]
v_{1,j-1}\\[4pt]
v_{1,j}
\end{array}
\right),
\quad
V_2=
\left(
\begin{array}{l}
v_{2,j-2}\\[4pt]
v_{2,j-1}\\[4pt]
v_{2,j}\end{array}
\right),
\quad
\ldots, 
\quad
V_{n}=\left(
\begin{array}{l}
v_{n,j-2}\\[4pt]
v_{n,j-1}\\[4pt]
v_{n,j}
\end{array}
\right).
$$
One then can show (cf. \cite{MOT}) that this construction provides the isomorphism
from Proposition~\ref{Isop}.
This also shows the 2-friezes as $\SL_3$-tilings, see \cite{BeRe}.

\begin{ex}
Considering the 2-frieze \eqref{SixFive} (completed with two rows of 0s above and below), 
we obtain the following hexagon
$$
\left(
\begin{array}{c}
1\\
0\\
0
\end{array}
\right), \quad
\left(
\begin{array}{c}
1\\
1\\
0
\end{array}
\right), \quad
\left(
\begin{array}{c}
2\\
6\\
1
\end{array}
\right), \quad
\left(
\begin{array}{c}
1\\
4\\
1
\end{array}
\right), \quad
\left(
\begin{array}{c}
0\\
1\\
1
\end{array}
\right), \quad
\left(
\begin{array}{c}
0\\
0\\
1
\end{array}
\right).
$$
\end{ex}

Thanks to the above geometric interpretation of 2-friezes, we have a useful geometric
formulas for the entries of a 2-frieze in terms of the corresponding $n$-gon.

\begin{prop}\cite{MOT}\label{Entry}
Given a closed $2$-frieze $(v_{i,j})$ of width $m=n-4$ and its corresponding $n$-gon
$(V_i)$, the entries of the frieze are given by 
\begin{equation*}
v_{i-\half,j-\half}=
\det(
V_{i-1}, V_i, V_{j-3}
),
\qquad
v_{i,j}=\det(
V_{j-3}, V_{j-2}, V_i
).
\end{equation*}
\end{prop} 
Note that in the sequel, we simplify the notation by using $|\cdot,\cdot,\cdot|=\det(\cdot,\cdot,\cdot)$.

Finally, we will also need the following linear recurrence relation.

\begin{prop}\cite{MOT}
The points of the $n$-gon $(V_i)$ satisfy  the following linear recurrence relation:
\begin{equation}
\label{Recur}
V_i=v_{i,i}\,V_{i-1}-v_{i-\half,i-\half}\,V_{i-2}+V_{i-3}.
\end{equation}
\end{prop}

\subsection{Connecting two 2-friezes}
We give examples of the operations of gluing of friezes described in Theorem \ref{GluThm}.
Note that this gluing preserves more than the two columns of the initial friezes. 
Triangular fragments of the old friezes appear in the new frieze (in the array \eqref{frag} the white bullets, resp. black bullets, stand for the initial entries in the top frieze, resp.bottom frieze, that still appear in the new frieze). 
The gluing can be described in terms of gluing of diagonals, 
starting from or ending at the pair 1 1,
instead of gluing of columns, at the top and bottom of the pair 1 1. 

\begin{equation}\label{frag}
\begin{array}{rccccccccccl}
\cdots&1&1&1&1 & 1 & 1 &1&1&1&1&\cdots\\
&& \circ & \circ  & \circ &\circ&\circ& \circ & \circ& \circ& & \\
&&& \circ & \circ  &\circ  & \circ &  \circ & \circ &&&\\
&&&&  \circ & \circ & \circ & \circ & &&\\
&&& && 1 & 1 &&&&&\\
 &&&&\bullet&\bullet &\bullet&\bullet&&&&\\
&&&\bullet &\bullet& \bullet & \bullet&\bullet &\bullet&&&\\
\cdots&1&1&1&1 & 1 & 1 &1&1&1& 1&\cdots
\end{array}
\end{equation}

\begin{ex}
(a)
The friezes \eqref{SixThree}-\eqref{SixFive} are all obtained as a gluing of the trivial frieze

$$
\begin{array}{cccccccccccccccc}
\cdots &1 & 1 & 1 & 1& 1& 1& 1& \cdots\\[6pt]
\cdots &1& 1& 1& 1& 1& 1& 1& \cdots \\
\end{array}
$$
 and the unique frieze of width 1
$$
\begin{array}{cccccccccccccccc}
\cdots &1 & 1 & 1 & 1& 1& 1& 1& \cdots\\[6pt]
\cdots & 1 & 1&   2&  3& 2&1& 1&\cdots \\[6pt]
\cdots &1& 1& 1& 1& 1& 1& 1& \cdots \\
\end{array}
$$

(b)
The 2-frieze in Figure \ref{ex2frieze} in Introduction, is obtained as the gluing of the frieze
of width 2 given in \eqref{SixOne} and the above unique frieze of width 1. 
This can be viewed for instance as follows, as the gluing of two columns, or equivalently two diagonals 
$$
\begin{array}{cc}
1&1\\
2&2\\
2&2\\
1&1
\end{array} \,+\,
\begin{array}{cc}
1&1\\
2&3\\
1&1
\end{array}\quad
\text{ or }
\quad
\begin{array}{ccccc}
1&1&&&\\
&2&2&&\\
&&2&2&\\
&&&1&1
\end{array} +
\begin{array}{cccc}
1&1&&\\
&3&2&\\
&&1&1
\end{array} 
\quad
\text{or}
\quad
\begin{array}{ccccc}
&&&1&1\\
&&2&2&\\
&2&2&\\
1&1&&&
\end{array} +
\begin{array}{cccc}
&&1&1\\
&1&2&\\
1&1&&
\end{array} .
$$

\end{ex}

\subsection{Cutting and gluing friezes}
We give below an example of Theorem \ref{Gluxy}.

\begin{ex}
The pair $(2,1)$ appears in both friezes \eqref{SixFive} and \eqref{SixTwo} of width 2.
The operation of Theorem \ref{Gluxy} gives the following new frieze
$$
\begin{array}{cccccccccccccccc}
1&1&1&1&1&1&1&1&1&1&1&1&1&1\\
4&3&1&3&13&5&\mathbf{1}&\mathbf{3}&\mathbf{5}&\mathbf{2}&2&6&4&2\\
2&1&8&10&2&8&14&\mathbf{2}&\mathbf{1}&8&10&2&8&14\\
3&5&2&2&6&4&\mathbf{2}&\mathbf{4}&\mathbf{3}&\mathbf{1}&3&13&5&1\\
1&1&1&1&1&1&1&1&1&1&1&1&1&1
\end{array}
$$
\end{ex}

\subsection{Proof  of Theorem \ref{GluThm}}
The initial two friezes are associated to the polygons $U=(U_1,\ldots , U_n)$,  $n=m+4$, and 
$V=(V_1, \ldots, V_k)$, $k=\ell+4$. 
The entries $(u_{i,j})$ in the first frieze and the entries $(v_{i,j})$ in the second frieze are given by
$$
\begin{array}{rcl}
u_{i,j}&=&|
U_{j-3}, U_{j-2}, U_i |\\[6pt]
u_{i+\half,j+\half}&=&
|U_{j-2}, U_{i}, U_{i+1}|
\end{array},
\qquad
\begin{array}{rcl}
v_{i,j}&=&|
V_{j-3}, V_{j-2}, V_i |\\[6pt]
v_{i+\half,j+\half}&=&
|V_{j-2}, V_{i}, V_{i+1}|
\end{array}
$$
We assume that the pair $1\; 1$ where the friezes are connected corresponds to
the entries $u_{4,n},\; u_{4+\half, n+\half}$ of the first frieze and to $v_{4,4}, \; v_{4+\half, 4+\half}$
in the second.
Define 
$$
U'_1:=U_1+|V_1,V_3,V_4| U_2
$$
Since the polygons are defined up to the action of $\SL(3,\C)$, one can assume that the following two sequences of consecutive vertices are the same:
$$V_{1}=U'_1,\;V_{2}=U_{2},\;V_3=U_n.$$
One considers the $n+k-3$-gon $(W_{i})$ consisting in 
$$
W_1=U'_1, \;W_2=U_2,  \;\ldots, W_{n-1}=U_{n-1},\; W_n=V_3, \;W_{n+1}=V_4,\ldots, W_{n+k-3}=V_{k},$$
see Figure \ref{glu11}.
\begin{figure}[hbtp]
\includegraphics[width=6.5cm]{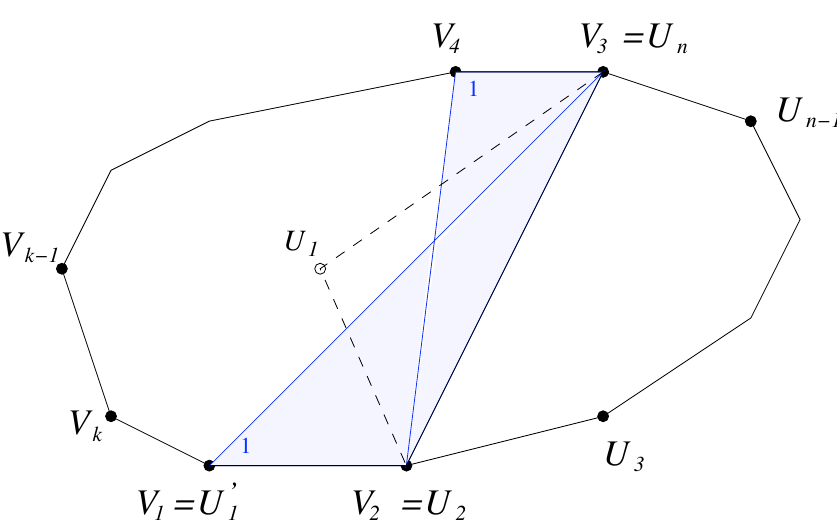}
\caption{Geometric situation when gluing two columns of friezes over 11}
\label{glu11}
\end{figure}

The change from $U_1$ to $U'_1$ has been made  in order to  have
$$1=|U_{n-1},V_3,V_4 |=|W_{n-1},W_n,W_{n+1}|.$$
Thus, any three consecutive vertices of the polygon $(W_i)$ form a matrix of determinant 1.
The 2-frieze associated to $(W_{i})$ is exactly
the glued frieze obtained in Theorem \ref{GluThm} (the pair $1 1$ corresponds now to 
the entries $w_{4,n}=|W_1,W_2,W_n|$ and $w_{4+\half,n+\half}=|W_2,W_n, W_{n+1}|$).

It remains to show that the entries are actually positive integers.

Since we already have two consecutive columns with positive entries
the positivity of the entire frieze $(w_{i,j})$ is guarantee by the local rule:
$$
 \begin{array}{ccccccc}
 *&B&*\\
 A&E&D\\
 *&C&*
\end{array}\quad
\Longrightarrow
\quad
D=(E+BC)/A.
$$
The vertices in $(V_i)$ and $(U_i)$ satisfy recurrence relations of the form
$$
V_{i}=a_iV_{i-1}-b_iV_{i-2}+V_{i-3},\;\; U_{i}=c_iU_{i-1}-d_iU_{i-2}+U_{i-3},
$$ 
where $a_i,b_i,c_i,d_i$ are integers.
Thus, by induction each vertex $V_i$  is a linear combination with integer coefficients of the first three points $V_1,V_2,V_3$. And similarly,  each vertex $U_i$  is a linear combination with integer coefficients of the three points $U_n,U_1,U_2$ and therefore of the three points  $U_n,U'_1,U_2$, which are the same as $V_3,V_1,V_2$.
One deduces that each vertex of $(W_i)$ is a linear combination of 
$V_1,V_2,V_3$ with integer coefficients. 
It follows, using the determinantal formulas of Proposition \ref{Entry}, that the entries in the frieze associated to  $(W_i)$ are all integers.

\subsection{Proof  of Theorem \ref{CutThm}}
The initial 2-frieze pattern \eqref{BigPat}
corresponds to some $n$-gon $V=(V_1,\ldots,V_n)$ in $\R^3$.
The entries of the 2-frieze are given by:
$$
v_{i-\half,j-\half}=
\left|
V_{i-1}, V_i, V_{j-3}
\right|,
\qquad
v_{i,j}=\left|
V_{j-3}, V_{j-2}, V_i
\right|,
$$
for integer $i,j\leq n-4$.
We fix $i,j$ such that $v_{i-\half,j-\half}=x$, $v_{i,j}=y$, $v_{i-\half,j-1}=u$, $v_{i+\half,j-\half}=v$.

Let us show that the condition \eqref{Cond} is necessary and sufficient
for the existence of a point $W$, such that the polygon
\begin{equation}\label{gonW}
V_1,V_2,\ldots,V_i,W,V_{j-3},V_{j-2},\ldots,V_n
\end{equation}
defines a positive integer 2-frieze pattern. 
As before, the positivity of the 2-frieze is guarantee by the positvity of two consecutive columns.
\begin{figure}[hbtp]
\includegraphics[width=4.5cm]{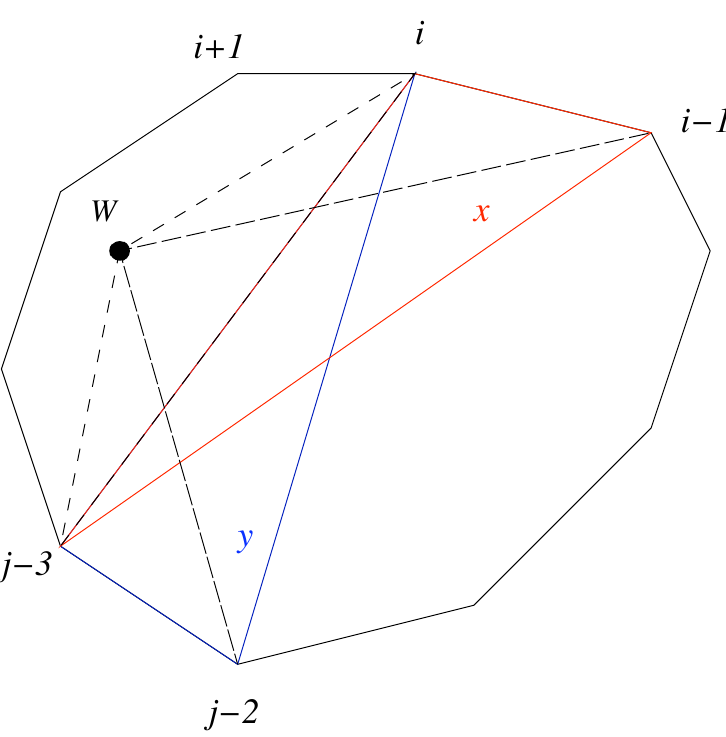}
\caption{Cutting a 2-frieze above $x,y$.}
\label{FirstFig}
\end{figure}

We can express the point $W$ as a linear combination of three
consecutive points:
$$
W=aV_i-bV_{i-1}+cV_{i-2}.
$$
The sequence \eqref{gonW} defines a closed 2-frieze if and only if $a,b,c$ are positive integers and 
$$
\left|
V_{i-1}, V_i, W
\right|=
\left|
V_i, W, V_{j-3}
\right|=
\left|
W, V_{j-3}, V_{j-2}
\right|=1.
$$
The first condition gives immediately
$c=1$. 
The second determinant can be written using the recurrence relation
$V_{i+1}=a_{i+1}V_i-b_{i+1}V_{i-1}+V_{i-2}$,
$$
\left|
V_i, W, V_{j-3}
\right|=|V_i\,,\,-bV_{i-1}+V_{i-2}\,,\,V_{j-3}|=bx+|V_i, b_{i+1}V_{i-1}+V_{i+1},V_{j-3}|=bx-b_{i+1}x+v.
$$
Hence, the condition $
\left|
V_i, W, V_{j-3}
\right|=1$ leads to
$$
b=b_{i+1}+\frac{1-v}{x}.
$$
So,  $v\equiv 1 \mod x$ is a necessary and sufficient condition. 
By symmetry we deduce similarly  $u\equiv 1 \mod y$.

\subsection{Proof of Theorem \ref{Gluxy}}
The condition \eqref{Cond2} is necessary because of Theorem \ref{CutThm}.
Let us prove that \eqref{Cond2} is also sufficient.
The two initial friezes are associated to polygons, the $n=m+4$-gon $(U_i)$ 
and the $k=\ell+4$-gon $(V_i)$. 
Assume that the entries appearing in the two friezes are given by
\begin{equation}\label{rsxyuv}
\begin{array}{rcl}
r&=&
|U_{2}, U_{n-2}, U_{n-1} |\\[6pt]
s&=&
|U_{2}, U_{3}, U_{n-1}|\\[6pt]
x&=&|
U_{1}, U_{2}, U_{n-1} |\\[6pt]
y&=&
|U_{2}, U_{n-1}, U_{n}|
\end{array},
\qquad
\begin{array}{rcl}
x&=&
|V_1,V_2,V_4 |\\[6pt]
y&=&
|V_{2},V_4,V_5|\\[6pt]
u&=&|
V_1,V_4,V_5 |\\[6pt]
v&=&
|V_1,V_2,V_5|
\end{array}
\end{equation}
Define 
$$
U'_1:=U_1+\a U_2
$$
where $\a$ is given by $u=1+\a y$ from \eqref{Cond2}.

\begin{figure}[hbtp]
\includegraphics[width=6.5cm]{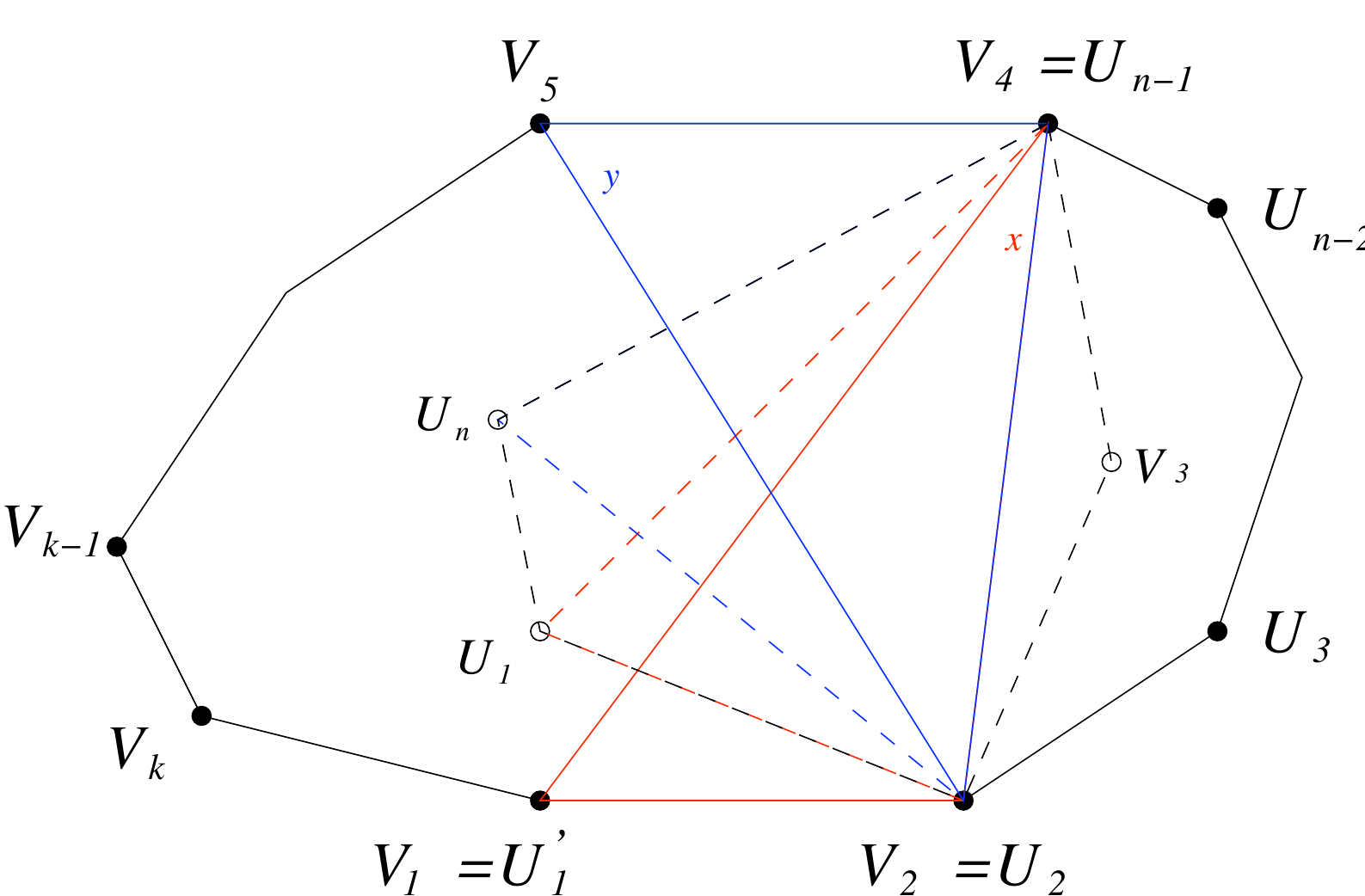}
\caption{Gluing friezes/polygons over x,y}
\label{Figxy}
\end{figure}

We are allowed to identify the following vertices
$$
V_1=U'_1,\;V_2=U_2, V_4=U_{n-1}.
$$
Now we consider the $n+k-5$-gon $(W_i)$ consisting in
$$
W_1=U'_1, W_2=U_2, \ldots, W_{n-2}=U_{n-2}, W_{n-1}=V_4, \ldots, W_{n+k-5}=V_k,
$$
see Figure \ref{Figxy}.
The choice of $U'_1$  has been made to guarantee 
$$
1=|U_{n-2},U_{n-1}, V_5|=|W_{n-2}, W_{n-1}, W_{n}|
$$
Indeed,
using the formula \eqref{rsxyuv} one can express
\begin{eqnarray*}
V_5&=&\frac{1}{x}\,(vV_4-uV_2+yV_1)\\[6pt]
&=&\frac{1}{x}\,(vU_{n-1}+(\a y -u)U_2+yU_1)
\end{eqnarray*}
and then compute
\begin{eqnarray*}
|U_{n-2},U_{n-1}, V_5|&=&\frac{\a y -u}{x}\,|U_{n-2},U_{n-1},U_2|+\frac{y}{x}\,|U_{n-2},U_{n-1},U_1|\\[6pt]
&=& \frac{(\a y -u)r}{x}+\frac{y}{x}\frac{x+r}{y}\\[4pt]
&=&1,
\end{eqnarray*}
where we used the fact that $|U_{n-2},U_{n-1},U_1|$ is equal to $(x+r)/y$, since it is the entry at the left of $x$ in the first frieze.

It follows that  any three consecutive vertices of the polygon $(W_i)$ form a matrix of determinant~1.
The 2-frieze associated to $(W_{i})$ is exactly
the glued frieze obtained in Theorem \ref{Gluxy}.

\section{A classification problem and other open questions}\label{Conc}

We have described several natural procedures to construct integral closed 2-friezes.
The first one is based on evaluation of cluster variables in the cluster algebra of type $A_2\times A_m$.
The friezes obtained this way are called unitary. 

We have proved that for $m\geq5$, there exist infinitely many integral closed 2-friezes
with $m$ non-trivial rows. 
This is due to the fact that in this case the cluster algebra is of infinite type and therefore one can
construct infinitely many unitary friezes.
We also proved that \textit{not every integral $2$-frieze is unitary}. 
Proposition \ref{unitfrieze} and Remark \ref{width3} provide examples of non-unitary friezes.
The problem of classification of integral closed 2-friezes (see Question \ref{classif})
can now be reformulated:

\begin{ques}
How many non-unitary integral $2$-friezes are there for a given width?
\end{ques}

It is also natural to ask the following.
Are the integral 2-friezes in bijective correspondence with a set of combinatorial objects similar to triangulations?
This would give an analog of the property (CC2) mentioned in the introduction.

The second type of constructions introduced in the paper 
comes from the geometric interpretation of 2-friezes in terms of polygons in the space.
It would be interesting to determine whether the set of unitary friezes 
is stable under the operations \eqref{glucol}, \eqref{BigPat}, \eqref{glufrixy}.

In a more algebraic setting, let us denote by $\F_m$ the set of integral closed 2-friezes of width $m$, 
and let us consider the vector space of basis vectors $\lbrace \F_m, m\geq 1\rbrace $ over an arbitrary field $K$
$$
\Ab=\bigoplus_{m\geq 1}K \F_m. 
$$ 
The operation \eqref{glucol} gives a structure of associative algebra on $\Ab$. 
This algebra is graded by $\deg(F)=m-1$ for $F\in\F_m$.
It could be interesting to study the algebra $\Ab$, for instance to determine generators and relations.

\bigskip

\noindent \textbf{Acknowledgements}.
Many ideas and results in the present paper were discussed with Valentin Ovsienko.
It is a pleasure to thank the CIRM, that offered us a Recherche en Bin\^ome stay.
It is also a pleasure to thank Brown University for its warm hospitality and the chance it offered me to 
discuss friezes and related topics
with V.Ovsienko, R.Schwartz and S.Tabachnikov.
My special gratitude goes to V. Fock who gave the first idea for the proof of Theorem~\ref{inffrieze},
and to the anonymous referee for the various valuable comments.


\end{document}